\documentclass[amscd,amssymb]{amsart}

\usepackage{amssymb}
\usepackage[all]{xy}

\newtheorem{thm}{Theorem}[section]
\newtheorem{lem}[thm]{Lemma}
\newtheorem{cor}[thm]{Corollary}
\newtheorem{prop}[thm]{Proposition}
\newtheorem{conj}[thm]{Conjecture}
\newtheorem{defn/prop}[thm]{Definition/Proposition}

\theoremstyle{definition}

\newtheorem{defn}[thm]{Definition}

\newtheorem{notation}[thm]{Notation}
\newtheorem{ex}[thm]{Example}

\theoremstyle{remark}

\newtheorem{rem}[thm]{Remark}

\numberwithin{equation}{section}

\newcommand{\thmref}[1]{Theorem~\ref{#1}}
\newcommand{\corref}[1]{Corollary~\ref{#1}}
\newcommand{\secref}[1]{\S\ref{#1}}
\newcommand{\conjref}[1]{Conjecture~\ref{#1}}
\newcommand{\propref}[1]{Proposition~\ref{#1}}
\newcommand{\lemref}[1]{Lemma~\ref{#1}}

\newcommand{\remref}[1]{Remark~\ref{#1}}

\newcommand{\Hom}{\operatorname{Hom}}
\newcommand{\End}{\operatorname{End}}
\newcommand{\Epi}{\operatorname{Epi}}

\newcommand{\TransRep}{\operatorname{TransRep}}

\newcommand{\hofib}{\operatorname*{hofib}}

\newcommand{\hoNat}{\operatorname{hoNat}}

\newcommand{\A}{{\mathcal  A}}

\newcommand{\U}{{\mathcal  U}}
\newcommand{\K}{{\mathcal  K}}
\newcommand{\V}{{\mathcal  V}}

\newcommand{\HH}{{\mathcal  H}}

\newcommand{\R}{{\mathbb  R}}

\newcommand{\C}{{\mathbb  C}}
\newcommand{\Z}{{\mathbb  Z}}

\newcommand{\Sinfty}{\Sigma^{\infty}}
\newcommand{\Oinfty}{\Omega^{\infty}}

\newcommand{\sm}{\wedge}

\newcommand{\ra}{\rightarrow}
\newcommand{\xra}{\xrightarrow}
\newcommand{\la}{\leftarrow}
\newcommand{\xla}{\xleftarrow}

\newcommand{\hra}{\hookrightarrow}

\newcommand{\im}{\operatorname{im}}

\begin{document}

\title[First Differential Conjecture]{The Whitehead Conjecture, the Tower of $S^1$ Conjecture, and Hecke algebras of type A}

\author[Kuhn]{Nicholas J.~Kuhn}
\address{Department of Mathematics \\ University of Virginia \\ Charlottesville, VA 22904}
\email{njk4x@virginia.edu}
\thanks{This research was partially supported by National Science Foundation grant 0967649.}
\dedicatory{Dedicated to the memory of my father Harold Kuhn (1925--2014) and uncle Leon Henkin (1921--2006), mathematicians and lifelong advocates for social justice.}
\date{August 6, 2014.}


\begin{abstract}  In the early 1980's the author proved G.~ W.~ Whitehead's conjecture about stable homotopy groups and symmetric products.  In the mid 1990's, Arone and Mahowald showed that the Goodwillie tower of the identity had remarkably good properties when specialized to odd dimensional spheres.

In this paper we prove that these results are linked, as has been long suspected.  We give a state-of-the-art proof of the Whitehead conjecture valid for all primes, and simultaneously show that the identity tower specialized to the circle collapses in the expected sense.

Key to our work is that Steenrod algebra module maps between the primitives in the mod $p$ homology of certain infinite loopspaces are determined by elements in the mod $p$ Hecke algebras of type $A$.  Certain maps between spaces are shown to be chain homotopy contractions by using identities in these Hecke algebras.

\end{abstract}

\maketitle

\section{Introduction and main results} \label{introduction}

G.~ W.~ Whitehead's conjecture about $p$--local stable homotopy groups and symmetric products of spheres was proved by the author \cite{kuhn1} when $p=2$, and by the author with S.~ Priddy \cite{kuhnpriddy} for all primes $p$.  The conjecture concerns {\em stable} homotopy groups, but the proof involves {\em unstable} methods.

The Goodwillie homotopy calculus tower of the identity functor on spaces, when evaluated on odd dimensional spheres, was shown to have many lovely properties by G.~ Arone and M.~ Mahowald \cite{aronemahowald}, with further properties evident when one views this as a Weiss orthogonal calculus tower.  One is left with a wonderfully efficient tool for studying the {\em unstable} homotopy groups of spheres using {\em stable} calculations.

It has long been suspected that delooped attaching maps for fibers in the tower for $S^1$ could be used to construct contracting retractions for the sequence arising in the Whitehead conjecture.  In this paper we show that this is true: we give a state-of-the-art proof of the Whitehead conjecture, and simultaneously show that the tower for $S^1$ collapses as expected.

The crucial observation for us is that the Steenrod algebra module maps between the primitives in the mod p homology of the various key infinite loopspaces are very limited, and determined by elements in the Hecke algebras associated to the family of Chevalley groups $GL_n(\Z/p)$.  Identities in these Hecke algebras, already identified in \cite{kuhnpriddy}, ultimately show that maps arising from the tower induce on homology primitives a chain homotopy contraction of the complex induced by the maps arising in the resolution of the Whitehead conjecture.

\subsection{A family of co-H-spaces}

We make the following conventions: all spaces and spectra are $p$--complete, homology and cohomology are with $\Z/p$--coefficients, and, if $Z$ is a based space, $QZ = \Oinfty \Sinfty Z$.

Let $\rho_k$ be the $p^k$ dimensional real permutation representation of the symmetric group $\Sigma_{p^k}$, which we regard as the group of permutations of the vectors in $E_k = (\Z/p)^k$. $E_k$ can be viewed as the translation subgroup in $\Sigma_{p^k}$, and has normalizer $E_k \rtimes GL_k(\Z/p)$.

As $\rho_k$ contains a one dimensional trivial module, the Thom space $BE_k^{\rho_k}$ is a suspension with an action of $GL_k(\Z/p)$. We then define $L_1(k)$ to be the natural retract of this space
associated to the Steinberg idempotent $e_k \in \Z_{(p)}[GL_k(\Z/p)]$ \cite{steinberg}.  In particular, $L_1(k)$ is naturally a co-H-space. The first two of these are themselves suspensions, and not just retracts of suspensions: $L_1(0) \simeq S^1$ and $L_1(1) \simeq \Sigma B\Sigma_p$, both completed at $p$.

Below, we will describe $H^*(L_1(k))$ as a module over the mod $p$ Steenrod algebra $\A$.  For now, we just note that it is of finite type, it is $c(k)-1$ connected where $c(k) = 2p^k - 1 -k$, and $H_{c(k)}(L_1(k)) = \Z/p$.

This construction extends to $\V_{\C}$, the category of finite dimensional complex vector spaces with inner product.  Given $V \in \V_{\C}$, the Thom space $BE_k^{(\R \oplus V) \otimes_{\R} \rho_k}$ is a suspension with an action of $GL_k(\Z/p)$. We then let $L_1(k,V) = e_kBE_k^{(\R \oplus V) \otimes_{\R} \rho_k}$, so that $L_1(k) = L_1(k,{\bf 0})$.

\subsection{The main theorem}

We consider families of maps
\begin{equation*}
\xymatrix{
\dots \ar@<.5ex>[r] & QL_1(3) \ar@<.5ex>[r]^{d_2} \ar@<.5ex>[l] & QL_1(2) \ar@<.5ex>[r]^{d_1} \ar@<.5ex>[l]^{s_2} & QL_1(1) \ar@<.5ex>[r]^{d_0} \ar@<.5ex>[l]^{s_1} & QL_1(0) \ar@<.5ex>[l]^{s_0} \ar@<.5ex>[r]^{d_{-1}} & S^1 \ar@<.5ex>[l]^{s_{-1}} \\}
\end{equation*}
where $s_{-1}$ and $d_{-1}$ are the evident inclusion and infinite loop structure maps.

\begin{thm} \label{big thm}
Suppose the following conditions hold for  $j= k, k-1$.
\begin{enumerate}
\item[{$(1_j)$}] \ $d_j$ is an infinite loop map, and is nonzero on homology primitives in dimension $c(j+1)$.
\item[{$(2_j)$}] \ $s_j$ is the specialization to $V=\bf 0$ of a natural transformation $$s_j(V): QL_1(j,V) \ra QL_1(j+1,V),$$
and is nonzero on homology primitives in dimension $c(j+1)$.
\end{enumerate}
Then $d_k  s_k + s_{k-1} d_{k-1}: QL_1(k) \ra QL_1(k)$ will be a homotopy equivalence.
\end{thm}

The next two corollaries follow formally from \thmref{big thm}.  (If the `formal' reasoning for the first of these is unclear, see Appendix \ref{appendix 1}.)

\begin{cor}  \label{cor 1} Suppose the conditions of the theorem hold for all $k$.  If in addition, $d_{k-1}d_k = 0$ for all $k\geq 0$, then, for all spaces $Z$,
the following chain complex is exact:
$$ \cdots \xra{d_2} [Z, QL_1(2)] \xra{d_1} [Z, QL_1(1)] \xra{d_{0}} [Z,QL_1(0)] \xra{d_{-1}} [Z,S^1] \ra 0.$$
Indeed, there are product decompositions $QL_1(k) \simeq Y_k \times Y_{k-1}$ so that the sequence
$$ \dots \xra{d_2} QL_1(2) \xra{d_1} QL_1(1) \xra{d_0} QL_1(0) \xra{d_{-1}} S^1$$
identifies with the evident `exact' sequence
$$ \dots  \ra Y_2 \times Y_1 \ra Y_1 \times Y_0 \ra Y_0 \times Y_{-1} \ra Y_{-1}.$$
\end{cor}

\begin{cor}  \label{cor 2} Suppose the conditions of the theorem hold for all $k$, and let $Z$ be a suspension.  If the sequence of homomorphisms of abelian groups
$$ \cdots \xla{s_2} [Z, QL_1(2)] \xla{s_1} [Z,QL_1(1)] \xla{s_{0}} [Z,QL_1(0)] \xla{s_{-1}} [Z,S^1] \la 0$$
is a chain complex, then it is a contractible chain complex, and thus exact.
\end{cor}

In \secref{Whitehead Conjecture section} and \secref{circle tower conj section} below, we describe two interesting families of maps $\{d_k\}$ and $\{s_k\}$ which we will prove fulfill the hypotheses of the main theorem.  The two corollaries then verify two substantive conjectures:
\begin{itemize}
\item The Whitehead Conjecture.  This is a conjecture about stable homotopy popularized by G.W.Whitehead in the 1960's, and previously proved by the author (with S.Priddy at odd primes) in the early 1980's \cite{kuhn1,kuhnpriddy}.  Our work here unifies, simplies, and clarifies these previous proofs.
\item The Tower of $S^1$ Conjecture.  This is a conjecture from the mid 1990's about the behavior of the Goodwillie tower of the identity, specialized to $S^1$.  M. Behrens has recently proved this for the prime 2 \cite{behrens2} using the Lie operad algebra structure on the identity tower identified by G. Arone and M. Ching \cite{aroneching}, and relying heavily on calculational details from \cite{kuhn1}.  Our work here works for all primes, and avoids both the use of the Arone-Ching work and reliance on \cite{kuhn1}.
\end{itemize}

\subsection{The strategy of the proof of \thmref{big thm}}

We say a little bit about the proof of \thmref{big thm}.

Let $\A$ be the mod $p$ Steenrod algebra and $Z$ a based space. Then $H_*(Z)$  is a coalgebra in the category of right $\A$--modules (with operations in $\A$  lowering degree), and we let $PH_*(Z)$ denote the submodule of primitives:
$$ PH_*(Z) = \{ x \in H_*(Z) \ | \ \Delta_*(x) = x \otimes 1 + 1 \otimes x \}.$$
A based map $f: Z \ra Y$ induces an $\A$--module map $f_*: PH_*(Z) \ra PH_*(Y)$, and if this map on primitives is a monomorphism, then $f_*: H_*(Z) \ra H_*(Y)$ will be a monomorphism.  It follows that if $Z$ is a $p$--complete, simple space of finite type, then a self map $f: Z \ra Z$ will be a homotopy equivalence if the induced map on primitives is an isomorphism, or even just a monomorphism.  Thus \thmref{big thm} will follow from the next theorem.

\begin{thm} \label{prim thm}  Assume the hypotheses of \thmref{big thm}. Then $${d_k}_*  {s_k}_* + {s_{k-1}}_* {d_{k-1}}_*: PH_*(QL_1(k)) \ra PH_*(QL_1(k))$$ will be an isomorphism.
\end{thm}

Conditions $(1_k)$ and $(2_k)$ turn out to respectively determine $$d_{k*}: PH_*(QL_1(k+1)) \ra PH_*(QL_1(k))$$ and $$s_{k*}: PH_*(QL_1(k)) \ra PH_*(QL_1(k+1))$$ up to a nonzero scalar: see \corref{d cor} and \corref{s cor}.  This rigidity result for $s_{k}$ (which is {\em not} an infinite loop map) is a key new technical discovery: in analyzing $s_{k*}$ in \secref{sk section}, we show potential `lower' terms are zero with a Weiss orthogonal calculus argument, and we show potential `higher' terms are zero with an $\A$--module calculation.  The homomorphisms $d_{k*}$ and $s_{k*}$ are then described by means of certain elements in the Hecke algebras of type $A$ (over $\Z/p$): see \corref{d cor 2} and \corref{s cor 2}.  Identities in these Hecke algebras then imply that, on $PH_*(QL_1(k))$, $(d_ks_k)_*$ and $(s_{k-1}d_{k-1})_*$ will be nonzero scalar multiples of complementary idempotents: see \secref{final proofs section}.

These Hecke algebras show up for the following reason.  Let $B_n < GL_n(\Z/p)$ be the subgroup of upper triangular matrices. The Hecke algebra $\mathcal H_n$ of type $A_{n-1}$ is the algebra of natural endomorphisms of the functor sending a $GL_n(\Z/p)$--module $M$ to its $B_n$--invariants $M^{B_n}$.  With some modification in the case of odd primes, our right $\A$--modules of primitives are direct sums of direct summands of duals of the left $\A$--modules $H^*(BE_n^{\rho_n})^{B_n}$.  One has
\begin{thm} The natural homomorphism of algebras
$$\mathcal H_n \ra \End_{\A}(H^*(BE_n^{\rho_n})^{B_n})$$
is an isomorphism.
\end{thm}
See \thmref{End thm} for the variant we need at odd primes.

\subsection{The Whitehead Conjecture} \label{Whitehead Conjecture section}

The Whitehead conjecture concerns the natural filtration of the Eilenberg--MacLane spectrum $H\Z$, when viewed as the infinite symmetric product of the sphere spectrum $S$.  Classical calculations \cite{nakaoka} show that the filtration
$$ S=SP^1(S) \ra SP^p(S) \ra SP^{p^2}(S) \ra SP^{p^3}(S) \ra \cdots \ra SP^{\infty}(S) = H\Z$$
realizes the admissible sequence length filtration on $H^*(H\Z) = \A/\A\beta$.  This implies, in particular, that for all $k$,
$$ H_*(SP^{p^k}(S)) \ra H_*(SP^{p^{k+1}}(S))$$
is monic.

By contrast, listed as a `classical' problem attributed to G.W.Whitehead in a 1970 conference proceedings \cite[Conjecture 84]{1970AMSproblems} is the conjecture that for all $k$,
$$ \pi_*(SP^{p^k}(S)) \ra \pi_*(SP^{p^{k+1}}(S))$$
is zero for $*>0$.

Let $L(0) = S$ and $L(k) = \Sigma^{-k}SP^{p^k}(S)/SP^{p^{k-1}}(S)$ for $k\geq 1$.  By diagram chasing, the conjecture is equivalent to the exactness of the sequence
$$ \cdots  \ra L(3) \xra{\delta_2} L(2) \xra{\delta_1} L(1) \xra{\delta_0} L(0) \xra{\delta_{-1}} H\Z
$$
on $p$--local homotopy groups.  Here $\delta_{-1}$ is the inclusion of the bottom cell, and, for $k \geq 0$, $\delta_k$ is the evident connecting map.

S. Mitchell and Priddy \cite{MiP83} construct a homotopy equivalence $\Sigma L(k) \simeq \Sinfty L_1(k)$; Arone and W. Dwyer \cite{aronedwyer} give a slightly different proof.

Choosing such equivalences, let $d_k: QL_1(k+1) \ra QL_1(k)$ be the infinite loop map
$$ QL_1(k+1) \simeq  \Oinfty \Sigma L(k+1) \xra{\Oinfty \Sigma \delta_k} \Oinfty \Sigma L(k) \simeq QL_1(k).$$

\begin{lem}  \label{d_k lemma} $d_k$ satisfies condition $(1_k)$, and $d_{k-1}d_k = 0$.
\end{lem}

The only thing here that is not clear is that $d_k$ is nonzero on $H_{c(k+1)}$.  Using \thmref{big thm}, this will be proved by induction on $k$ in \secref{final proofs section}.

We will shortly describe two families of maps $s_k$ satisfying conditions $(2_k)$: see \lemref{s_k lemma} and \remref{Whitehead conjecture remark}.
As in \cite{kuhn3}, we call a spectrum {\em spacelike} if it is a retract of a suspension spectrum. \corref{cor 1}, applied to our maps $d_k$, can be restated as follows.

\begin{thm}[Whitehead Conjecture] \label{Whitehead conjecture}  \ If $Y$ is spacelike, then the following sequence is exact:
$$ \cdots \xra{\Sigma \delta_2} [Y,\Sigma L(2)] \xra{\Sigma \delta_1} [Y,\Sigma L(1)] \xra{\Sigma \delta_{0}} [Y,\Sigma L(0)] \xra{\Sigma \delta_{-1}} [Y,\Sigma H\Z] \ra 0.$$
\end{thm}

We have reproved the main theorems of \cite{kuhn1,kuhnpriddy}.  Since $L(0) = S^0$ and $L(1) \simeq \Sinfty B\Sigma_p$, exactness at $[Y, \Sigma L(0)]$ is a strengthened version (since it is once delooped) of the Kahn-Priddy Theorem \cite{kahnpriddy}, which resolved a conjecture of M. Mahowald from the same 1970 conference proceedings \cite[Conjecture 81]{1970AMSproblems}.

\begin{rem} One can also construct maps $\delta_k: \Sinfty L_1(k+1) \ra \Sinfty L_1(k)$ `by hand' such that $d_k = \Oinfty \Sigma \delta_k$ satisfies condition $(1_k)$.  Using that $L_1(k)$ is a retract of $\Sigma BE_{k+}$ \cite{MiP83}, one lets $\delta_k$ be the composite
$$ \Sinfty L_1(k+1) \ra \Sinfty \Sigma BE_{(k+1)+} \xra{tr} \Sinfty \Sigma BE_{k+} \ra \Sinfty L_1(k)$$
where $tr$ is the transfer associated to the subgroup inclusion $E_k < E_{k+1}$.  This is essentially what was done in \cite{kuhnpriddy}.  Showing that condition $(1_k)$ holds can then be given a direct proof, which we omit here.

In fact, all such maps `$d_k$' turn out to be the same, up to a unit in $\Z_p$, and $d_{k-1}d_k = 0$ holds automatically: one can compute that, for all $k$, $[L(k+1), L(k)]= \Z_p$ and $[L(k+2),L(k)] = 0$.  This follows from the Segal Conjecture for elementary abelian groups \cite{agm} and calculations with Steinberg idempotents in \cite{kuhn5} or \cite{nishida}.
\end{rem}

\subsection{The Tower of $S^1$ Conjecture} \label{circle tower conj section}

Given $V \in \V_{\C}$, results by Arone and coauthors Mahowald \cite{aronemahowald}, Dwyer \cite{aronedwyer}, and Kankaanrinta \cite{aronekan} combine to show that there is a strongly convergent tower of principal fibrations under $S^{\R \oplus V}$,
\begin{equation*}
\xymatrix{
S^{\R \oplus V} \ar[d] \ar[dr] \ar[drr]\ar[drrr]&&&&\\
QS^{\R \oplus V} = P_1(V) & P_{p}(V) \ar[l] & P_{p^2}(V) \ar[l] & P_{p^3}(V) \ar[l] & \cdots \ar[l]
}
\end{equation*}
with $\hofib \{P_{p^k}(V) \ra P_{p^{k-1}}(V)\} \simeq \Omega^k QL_1(k,V)$ for $k>0$.  This tower can be regarded as either the Weiss tower \cite{weiss} of the functor $V \rightsquigarrow S^{\R \oplus V}$, or the Goodwillie tower \cite{goodwillie3} of the identity functor on based spaces, specialized to odd dimensional spheres.

Using the former interpretation, Arone, Dwyer, and Lesh \cite{aronedwyerlesh} show that the connecting maps of the tower, $\Omega^k QL_1(k,V) \ra \Omega^k QL_1(k+1,V)$, can be naturally delooped $k$ times. Choosing deloopings, one sees that one has a sequence of natural maps
$$
S^{\R \oplus V} \xra{s_{-1}(V)} QL_1(0,V) \xra{s_0(V)}  QL_1(1,V)  \xra{s_1(V)}  QL_1(2,V) \ra \cdots,
$$
where $s_{-1}(V)$ is the natural inclusion $S^{\R \oplus V} \ra QS^{\R \oplus V}$.

Specializing to $V = \bf 0$ and letting $s_k = s_k(\bf 0)$, the sequence takes the form
$$
S^1 \xra{s_{-1}}  QL_1(0) \xra{s_0}   QL_1(1) \xra{s_1}  QL_1(2) \ra \cdots .
$$

\begin{lem}  \label{s_k lemma} $s_k$ satisfies condition $(2_k)$.
\end{lem}

This will be proved in \secref{final proofs section}.  As before, the only thing here that is not clear is that $s_k$ is nonzero on $PH_{c(k+1)}$, and, as before, this will be proved by induction on $k$, using \thmref{big thm}.

Note that, from the construction, each composite $s_ks_{k-1}$ is  null after applying $\Omega^k$.  \corref{cor 2} applied to our maps $s_k$ thus proves the following version of longtime conjectures about the tower of $S^1$.

\begin{thm}[The Tower of $S^1$ Conjecture] \label{circle tower conjecture} If, for some $k$, $Z$ is the $k$--fold suspension of a CW complex and $\dim Z < c(k+1)$, then the following sequence is exact:
$$0 \ra [Z,S^1] \xra{s_{-1}} [Z, QL_1(0)] \xra{s_0}  [Z, QL_1(1)]  \xra{s_1}  [Z,QL_1(2)] \ra \cdots .$$
\end{thm}

All spheres satisfy the hypothesis for $Z$, so one gets an exact sequence on homotopy groups.  \cite{behrens2} shows this is true when $p=2$.

We conjecture an addendum to \cite{aronedwyerlesh}, which would allow one to simplify the hypotheses on $Z$ in the last theorem to just the statement that $Z$ be a suspension.

\begin{conj} \label{loop prop}  All natural deloopings $s_k(V)$ satisfy $(\Omega s_k(V)) (\Omega s_{k-1}(V)) = 0$.
\end{conj}

We discuss our failed attempt to prove this in Appendix \ref{loop prop appendix}.  We do prove, however, that there is a unique $(k-1)$--fold delooping $\Omega s_k(V)$ of the $k$th connecting map.

If one applies $\pi_*$ to the tower for $S^{\R \oplus V}$, one gets a spectral sequence converging to $\pi_*(S^{\R \oplus V})$.  These spectral sequences are compatible as $V$ varies, as the map $S^{\R\oplus V} \ra \Omega^W S^{\R \oplus V \oplus W}$ induces a map of towers with the maps on fibers homotopic to standard quotient maps. \thmref{circle tower conjecture} implies that the spectral sequence for $S^1$ collapses (to the known answer!) at $E^2$, and so can be used to help determine $d_1$ differentials in the spectral sequences for the other odd dimensional spheres.  We note that $L_1(k,V)$ is $2p^k-2-k+2\dim_{\C}V$ connected, so the spectral sequences converge very quickly.  See \cite{behrens1} for some enhanced thoughts along these lines, and many calculations when $p=2$.

\begin{rem} \label{Whitehead conjecture remark} If one is only interested in the Whitehead Conjecture, there is an alternative set of maps available for the contracting homotopies, defined in the spirit of \cite{kuhn1}.  Let $s_k(V): QL_1(k,V) \ra QL_1(k+1,V)$ be given as the composite
$$ QL_1(k,V) \xra{j_p} QD_pL_1(k,V) \xra{\Oinfty \pi} QL_1(k+1,V)$$
where $D_pX$ is the $p$th extended power of $X$, $j_p$ is the $p$th James-Hopf map (see \secref{james hopf map subsection}), and $\pi$ is a stable retraction constructed as in \corref{DpLk retract cor}. With $s_k(V)$ so defined, $s_k = s_k(\bf 0)$ satisfies condition $(2_k)$.  See \remref{Whitehead conjecture remark 2} for more about using these maps $s_k$ to prove the Whitehead Conjecture.
\end{rem}

\subsection{Acknowledgements} I would like to thank Geoffrey Powell for a critical reading of the first version of this paper, and, in particular, for catching an incorrect, but thankfully inessential, lemma.

\section{Computing maps $f:QZ \ra QY$ on homology primitives}  The maps in our main theorem all have the form $f:QZ \ra QY$, where $Z$ and $Y$ are co-H-spaces.  In this section, we discuss how to compute such maps on homology primitives.

\subsection{Stable splittings and maps from $QZ$ to an infinite loop space} \label{james hopf map subsection}

For $r \geq 1$, the $r$th extended power of a based space $Z$ is the space $D_rZ = E\Sigma_{r+} \sm_{\Sigma_r}Z^{\sm r}$.  It is convenient to let $D_0Z=S^0$.  These constructions extend to spectra so that $D_r(\Sinfty Z) = \Sinfty D_r Z$ \cite{lmms}.

Let $\displaystyle DX = \bigvee_{r=0}^{\infty} D_rX$ and $\displaystyle D^{\Pi}X = \prod_{r=0}^{\infty} D_rX$.
The various subgroup inclusions
$$\Sigma_i \times \Sigma_j \leq \Sigma_{i+j} \text{ and } \Sigma_q \wr \Sigma_r < \Sigma_{qr}$$ induce natural maps
$$ D_i X \sm D_j X \ra D_{i+j}X \text{ and } D_qD_r X \ra D_{qr}X.$$
These in turn assemble to make $D$ into a monad, and an $E_{\infty}$ ring spectrum is an $D$--algebra in spectra.  Examples of interest to us are $DX$, $D^{\Pi}X$, and $\Sinfty (\Oinfty X)_+$, for $X$ a spectrum.

When $X$ is a suspension spectrum, these examples can be compared.  Stable splittings as in the next theorem go back to work of D. Kahn \cite{kahn splitting}, and modern presentations are given in \cite[Appendix B]{kuhn7} and \cite[Thm.2.2]{kuhn8}.
\begin{thm} \label{stable splitting thm}  There are natural stable maps
\begin{equation*} \label{stable splitting} D \Sinfty Z  \ra \Sinfty (QZ)_+ \ra D^{\Pi} \Sinfty Z
\end{equation*}
such that
\begin{itemize}
\item the composite is the canonical map from a coproduct to a product.
\item both maps are weak equivalences if $Z$ is path connected.
\item both maps are maps of $E_{\infty}$ ring spectra augmented over $S$.
\item the inclusion and projection maps $\Sinfty D_1 Z \ra \Sinfty QZ \ra \Sinfty D_1Z$ correspond to the unit and counit maps $\eta: Z \ra QZ$ and $\epsilon: \Sinfty QZ \ra \Sinfty Z$.
\end{itemize}
\end{thm}

\begin{notation} Given a space $Z$ and a spectrum $X$, the map $D \Sinfty Z  \ra \Sinfty (QZ)_+$
induces a homomorphism of homotopy groups
$$ [QZ, \Oinfty X] = [\Sinfty QZ,X] \ra \prod_{r=1}^{\infty} [\Sinfty D_rZ,X].$$
Under this homomorphism, we let the image of $f:QZ \ra \Oinfty X$ have components $f(r): \Sinfty D_rZ \ra X$, with adjoints $\tilde f(r): D_rZ \ra \Oinfty X$.  In particular, the first map $\tilde f(1)$ identifies with the composite $Z\xra{\eta} QZ \xra{f} \Oinfty X$.
\end{notation}

\begin{defn}  For $r\geq 1$, the $r$th James-Hopf map $j_r: QZ \ra QD_rZ$ is defined as the adjoint to the projection $\Sinfty QZ \ra \Sinfty D_rZ$.  (In particular, $j_1$ is the identity.)
\end{defn}

With this notation and definition, we have the following lemma.

\begin{lem}  \label{decomp lemma} If $Z$ is path connected, then any map $f: QZ \ra \Oinfty X$ decomposes as $$f \simeq \sum_{r=0}^{\infty} \Oinfty f(r) \circ j_r: QZ \ra \Oinfty X.$$
\end{lem}

\begin{proof}  Firstly we explain what the infinite sum means.  Convergence in $[QZ, \Oinfty X]$ is with respect to the topology induced by the subgroups
$$\ker \{[QZ, \Oinfty X] \ra \prod_{r=1}^{n} [\Sinfty D_rZ,X]\}.$$  The lemma now follows by noting that, by construction,
\begin{equation*}
(\Oinfty f(r) \circ j_r)(s) =
\begin{cases}
f(r) & \text{if } s = r \\ * & \text{if } s \neq r.
\end{cases}
\end{equation*}
\end{proof}

\begin{rem}  $f$ is an infinite loop map $\Leftrightarrow f= \Oinfty f(1) \Leftrightarrow f(r)$ is null for all $r>1$.
\end{rem}

When $Z$ is a co-H-space, the stable equivalence has two other nice properties.

\begin{prop}\cite{kuhn7} \label{diag prop}  Let $Z$ be a co-H-space. There is a commutative diagram
\begin{equation*}
\xymatrix{
D(\Sinfty Z) \ar[d]^{\wr} \ar[rr]^-{\delta} && D(\Sinfty Z) \sm D(\Sinfty Z) \ar[d]^{\wr}  \\
\Sinfty (QZ)_+ \ar[r]^-{\Delta} & \Sinfty (QZ \times QZ)_+  \ar[r]^-{\sim} &\Sinfty (QZ)_+ \sm \Sinfty (QZ)_+ }
\end{equation*}
where $\Delta$ is induced by the diagonal on $QZ$, and $\delta$ is the wedge over all $(i,j)$ of the transfer maps
$$D_{i+j}(\Sinfty Z) \ra D_i(\Sinfty Z) \sm D_j(\Sinfty Z)$$
associated to the inclusions $\Sigma_i \times \Sigma_j \leq \Sigma_{i+j}$.
\end{prop}

The main result of \cite{kuhn6}, as strengthened for co-H-spaces in \cite[Appendix B]{kuhn7}, says the following.

\begin{prop} \label{James map prop} If $Z$ is a co-$H$-space, then the composite
$$ \Sinfty D_rZ \hra \Sinfty QZ \xra{\Sinfty j_t} \Sinfty QD_tZ \twoheadrightarrow \Sinfty D_sD_tZ$$
is null unless $r=st$, when it equals the transfer
$$D_{st}(\Sinfty Z) \ra D_sD_t(\Sinfty Z)$$
associated to the subgroup inclusion $\Sigma_s \wr \Sigma_t < \Sigma_{st}$.
\end{prop}

\begin{rem}  Without the hypothesis that $Z$ is a co-H-space, the last two propositions still hold, modulo higher terms identified in \cite{kuhn6, kuhn7}.
\end{rem}
\subsection{$H_*(QZ)$, $H_*(DZ)$, and the functors $R_n$}

As the homology of $E_{\infty}$ spaces, both $H_*(\Oinfty X)$ and $H_*(DZ)$ admit homology operations. If $X = \Sinfty Z$, so that $\Oinfty X = QZ$, such structure will be compatible under the stable splitting of $QZ$.

Firstly, the H-space multiplication on $\Oinfty X$ and the maps $D_iZ \sm D_jZ \ra D_{i+j}Z$ respectively induce commutative associative products on $H_*(\Oinfty X)$ and $H_*(DZ)$.

Secondly, the maps $D_p(\Oinfty X_+) \ra \Oinfty X_+$ and $D_pD_rZ \ra D_{pr}Z$ induce Dyer--Lashof operations on $H_*(\Oinfty X)$ and $H_*(DZ)$.

We remind readers how this goes.

When $p=2$, the operations take the form
$$Q^i: H_n(\Oinfty X) \ra H_{n+i}(\Oinfty X) \text{ and } Q^i: H_n(D_rZ) \ra H_{n+i}(D_{2r}Z).$$
One can iterate; if $I=(i_1, \dots, i_n)$, one lets $Q^I = Q^{i_1} \dots Q^{i_n}$, and defines $l(I)$, the length of $I$, to be $n$.
These satisfy Adem relations and the unstable relation: $Q^ix=0$ if $i<|x|$. They interact with Steenrod operations via the Nishida relations, and interact with product structure via the Cartan formula and the restriction property: $Q^{i}x = x^2$ if $i=|x|$.

When $p$ is odd, the operations take the form
$$Q^i: H_n(\Oinfty X) \ra H_{n+2(p-1)i}(\Oinfty X) \text{ and } Q^i: H_n(D_rZ) \ra H_{n+2(p-1)i}(D_{pr}Z).$$
Given $(\epsilon_1, i_i, \dots, \epsilon_n, i_n)$ with each $\epsilon_j = 0$ or 1, one lets $Q^I = \beta^{\epsilon_1}Q^{i_1} \dots \beta^{\epsilon_n}Q^{i_n}$, and defines $l(I)=n$.
As at 2, one has Adem and Nishida relations, and the Cartan formula.  The unstable relation now reads $Q^ix=0$ if $2i<|x|$ and the restriction property says $Q^{i}x = x^p$ if $2i=|x|$.

Using this structure, if $Z$ is path connected, $H_*(QZ) \simeq H_*(DZ)$ is a well known functor of $H_*(Z)$ (see \cite{clm} or \cite[Thm.IX.2.1]{bmms}), as we now describe.

One defines functors
$$R_n: \text{ right $\A$--modules } \ra \text{ right $\A$--modules }$$
by letting $R_nM$ be the span of length $n$ Dyer--Lashof operations applied to $M$:
$$R_nM = \langle Q^Ix \ | \ l(I) = n,  x \in M\rangle/(\text{unstable and Adem relations}),$$
with Steenrod operations acting via the Nishida relations.

Let $R_*M = \bigoplus_{n=0}^{\infty} R_nM$, so $R_*M$ has a left action by the Dyer--Lashof algebra, and a compatible right action by $\A$.
Then $$H_*(QZ) \simeq H_*(DZ) \simeq U(R_*\tilde H_*(Z)),$$ where, if $N$ is a module over the Dyer--Lashof algebra, $U(N)$ is the free graded commutative algebra on $N$ subject to the relations $x^2 = Q^{i}x$, if $p=2$ and $i=|x|$, and $x^p = Q^{i}x$, if $p$ is odd and $2i=|x|$.

Note that $R_n\tilde H_*(Z)$ is naturally a sub $\A$--module of both $H_*(D_{p^n}Z)$ and $H_*(QZ)$, and these submodule embeddings are compatible under the isomorphism $H_*(DZ) \simeq H_*(QZ)$. The following two `geometric' constructions of $R_n\tilde H_*(Z)$ are both illuminating and useful.

\begin{lem} \label{Rn lem 1} $R_n\tilde H_*(Z)$ is the image in homology of the (stable) map $$\epsilon: \Sigma D_{p^n}(\Sigma^{-1} Z) \ra D_{p^n}Z$$ induced by the diagonal $S^1 \ra S^{p^n}$. \\
\end{lem}
\begin{proof} $\epsilon$ is well known to commute with Dyer--Lashof operations, and annihilate product decomposables. (See the discussion following \cite[Lemma 2.10]{kuhnmccarty} for how to find this in the literature.)
\end{proof}

\begin{lem} \label{Rn lem 2} $R_n\tilde H_*(Z)$ is the image in homology of the Steenrod diagonal
$$BE_{n+} \sm Z \ra D_{p^n}Z.$$
\end{lem}
\begin{proof}  By iteration, it suffices to show this when $n=1$.  Nishida proved this case: see  \cite[Theorems 1 and 2]{nishida68}.
\end{proof}

\begin{prop} \label{Rn trans prop} The transfer and inclusion maps associated to $ \Sigma_{p^n}\wr\Sigma_{p^k} < \Sigma_{p^{n+k}}$ restrict to defined the dotted arrows in the a commutative diagram
$$
\xymatrix{
R_{n+k}\tilde H_*(Z) \ar@{^(->}[d] \ar@{-->}[r]^-{tr} & R_{n}R_k\tilde H_*(Z) \ar@{^(->}[d] \ar@{-->}[r]^-{i_*} & R_{n+k}\tilde H_*(Z) \ar@{^(->}[d] \\
\tilde H_*(D_{p^{n+k}}Z) \ar[r]^-{tr} & \tilde H_*(D_{p^n}D_{p^k}Z) \ar[r]^-{i_*} & \tilde H_*(D_{p^{n+k}}Z)}
$$
\end{prop}
\begin{proof} That $i_*$ restricts as indicated is clear.  That the transfer $tr$ restricts is a consequence of \lemref{Rn lem 1}, noting that $\epsilon$ commutes with the transfer.
\end{proof}

\begin{cor}  Via $tr$ and $i_*$, the $\A$--module $R_{n+k}\tilde H_*(Z)$ is a direct summand in $R_nR_k\tilde H_*(Z)$.
\end{cor}
\begin{proof} $i_* \circ tr$ is multiplication by the index, and $[\Sigma_{p^n}\wr\Sigma_{p^k}:\Sigma_{p^{n+k}}] \equiv 1\mod p$.
\end{proof}

\subsection{Primitives in $H_*(\Oinfty X)$ and $H_*(DZ)$}

Both $H_*(\Oinfty X)$ and $H_*(DZ)$ become Hopf algebras with coproducts respectively induced by the diagonal
$$ \Delta: \Oinfty X \ra \Oinfty X \times \Oinfty X$$
and the stable map
$$ \delta: D(\Sinfty Z) \ra D(\Sinfty Z) \sm D(\Sinfty Z)$$
as in \propref{diag prop}.

The Cartan formula shows that, if $x$ is a primitive in either $H_*(\Oinfty X)$ or $H_*(DZ)$, then so is $Q^Ix$ for any $I$.  It follows that the Dyer-Lashof module structure maps
$$ \Theta_*: R_*H_*(\Oinfty X) \ra H_*(\Oinfty X) \text{ and } \Theta_*:R_*H_*(DZ) \ra H_*(DZ)$$
restrict to maps
$$ \Theta_*:R_*PH_*(\Oinfty X) \ra PH_*(\Oinfty X) \text{ and } \Theta_*:R_*PH_*(DZ) \ra PH_*(DZ).$$

For any $Z$, $\tilde H_*(Z)$ will be primitive in $H_*(DZ)$, and it follows that
$$ PH_*(DZ) = R_*\tilde H_*(Z).$$
If $Z$ is also a co-H-space, so that $\tilde H_*(Z)$ is primitive in $H_*(QZ)$, then also
$$ PH_*(QZ) = R_*\tilde H_*(Z).$$

(Note that these statements are compatible with \propref{diag prop}.)

\subsection{Computing maps $f:QZ \ra \Oinfty X$ on homology primitives}

The results in the previous subsections now combine to tell us how to compute
$$ f_*: PH_*(QZ) \ra PH_*(\Oinfty X)$$
if $Z$ is a co-H-space, and $f: QZ \ra \Oinfty X$ is an arbitrary continuous map.

\begin{lem} If $Z$ is a co-H-space, so is the space $D_rZ$ for $r\geq 1$.
\end{lem}
\begin{proof}  The reduced diagonal on $D_rZ$ factors as
$$ D_r Z \xra{D_r\Delta} D_r(Z \sm Z) \ra D_r Z \sm D_rZ,$$
and the first of these maps is null if $Z$ is a co-H-space.
\end{proof}

\begin{thm} \label{QZ to Oinfty X on primitives thm} If $Z$ is a co-H-space and $f: QZ \ra \Oinfty X$ is a continuous map, then the induced map on homology primitives can be computed as
$$ f_* = \sum_{m=0}^{\infty} \Oinfty f(p^m)_* \circ j_{p^m*}: R_*\tilde H_*(Z) = PH_*(QZ) \ra PH_*(\Oinfty X).$$
Furthermore, each term $\Oinfty f(p^m)_* \circ j_{p^m*}$ is determined by the Dyer-Lashof operations on $H_*(\Oinfty X)$, together with the map
$$ R_m \tilde H_*(Z) \subset \tilde H_*(D_{p^m}Z) \xra{\tilde f(p^m)_*} PH_*(\Oinfty X),$$
as follows: on $R_{n+m} \tilde H_*(Z)$, it is zero for $n<0$, and, for $n \geq 0$, it is the composite
$$ R_{n+m}\tilde H_*(Z) \xra{tr}R_{n}R_m\tilde H_*(Z) \xra{R_n\tilde f(p^m)_*} R_nPH_*(\Oinfty X) \xra{\Theta_n} PH_*(\Oinfty X).$$
\end{thm}
\begin{proof}  By \lemref{decomp lemma},
$$f_* = \sum_{r=1}^{\infty} \Oinfty f(r)_* \circ j_{r*}: R_*\tilde H_*(Z) = PH_*(QZ) \ra PH_*(\Oinfty X).$$
As $R_m\tilde H_*(Z) \subset \tilde H_*(D_{p^m}Z)$, \propref{James map prop} then tells us that $j_{r*}$, restricted to $PH_*(QZ)$, will be zero unless $r=p^m$ for some $m$.  This establishes the first statement of the theorem.

To prove the rest of the theorem, we first note that, by the lemma, $D_{p^m}Z$ is a co-H-space, and thus
$$\tilde f(p^m): D_{p^m}Z \ra \Oinfty X$$
really does induce a map
$$ \tilde f(p^m)_*: \tilde H_*(D_{p^m}Z) \ra PH_*(\Oinfty X),$$
as implied in the statement of the theorem.

Since $\Oinfty f(p^m): QD_{p^m} Z \ra \Oinfty$ is the infinite loop extension of $\tilde f(p^m)$,
$$\Oinfty f(p^m)_*: PH_*(QD_{p^m}Z) \ra PH_*(\Oinfty X)$$ is the sum over $n$ of the composites
$$R_{n}\tilde H_*(D_{p^m}Z) \xra{R_n\tilde f(p^m)_*} R_nPH_*(\Oinfty X) \xra{\Theta_n} PH_*(\Oinfty X).$$
Meanwhile, \propref{James map prop} and \propref{Rn trans prop} combine to say that
$$ R_{m+n}\tilde H_*(Z) \subset PH_*(QZ) \xra{j_{p^m*}} PH_*(QD_{p^m}Z)$$
is zero if $n<0$, and is the composite
$$ R_{m+n}\tilde H_*(Z) \xra{tr} R_nR_m\tilde H_*(Z) \subset R_n \tilde H_*(D_{p^m}Z) \subset PH_*(QD_{p^m}Z)$$
if $n\geq 0$.

\end{proof}

\begin{cor} \label{QZ to QY on primitives cor} If $Y$ and $Z$ are co-H-spaces and $f: QZ \ra Q Y$ is a continuous map, then the induced map on homology primitives can be computed as
$$ f_* = \sum_{m=0}^{\infty} \Oinfty f(p^m)_* \circ j_{p^m*}: R_*\tilde H_*(Z) = PH_*(QZ) \ra PH_*(Q Y) = R_*\tilde H_*(Y).$$
Furthermore, each term $\Oinfty f(p^m)_* \circ j_{p^m*}$ is determined by
$$ R_m \tilde H_*(Z) \subset \tilde H_*(D_{p^m}Z) \xra{\tilde f(p^m)_*} PH_*(Q Y) = R_*\tilde H_*(Y)$$
as follows: on $R_{n+m} \tilde H_*(Z)$, it is zero for $n<0$, and, for $n \geq 0$, it is the composite
$$ R_{n+m}\tilde H_*(Z) \xra{tr}R_{n}R_m\tilde H_*(Z) \xra{R_n\tilde f(p^m)_*} R_nR_*\tilde H_*(Y) \xra{i_*} R_{n+*}\tilde H_*(Y).$$
\end{cor}

\section{Proof of \thmref{prim thm}: first reductions}

It is convenient to let $L_k = \tilde H_*(L_1(k))$, so, in particular, $L_0 = \tilde H_*(S^1)$.

As $PH_*(QL_1(k)) = R_*L_k$, to prove \thmref{prim thm}, we need to understand the maps

$$ R_*L_{k+1} \begin{array}{c} d_{k*} \\[-.08in] \longrightarrow \\[-.1in] \longleftarrow \\[-.1in] s_{k*}
\end{array} R_*L_{k} \begin{array}{c} d_{{k-1}*} \\[-.08in] \longrightarrow \\[-.1in] \longleftarrow \\[-.1in] s_{{k-1}*}
\end{array} R_*L_{k-1}.$$

In this section, we show how the hypotheses of \thmref{prim thm} fit with the theory of the last section to greatly simplify the problem.

\subsection{Analyzing $d_{k*}$}

As $d_k$ is assumed to be an infinite loop map, $d_{k*}$ is determined by $\tilde d_{k}(1)_*: L_{k+1} \ra R_*L_k$, where $\tilde d_{k}(1)$ is the composite
$$ L_1(k+1) \xra{\eta} QL_1(k+1) \xra{d_k} QL_1(k).$$
The following three lemmas restrict what $\tilde d_{k}(1)_*$ could be.

\begin{lem} $\Hom_{\A}(L_{k+1},L_k) = 0$
\end{lem}
This was given an elementary proof by the author in \cite{kuhn5}.

\begin{lem} \label{d lem 1} $\Hom_{\A}(L_{k+1},R_nL_k) = 0$ if $n>1$.
\end{lem}

This is a special case of \thmref{m not n thm}.

\begin{lem} \label{d lem} $\Hom_{\A}(L_{k+1},R_1L_k) \simeq \Z/p$.
\end{lem}

See \lemref{d lem 2} for a more precise statement.

\begin{cor}  \label{d cor} If $d_k$ satisfies condition $(1_k)$, then
$d_{k_*}$ is a composite of the form
$$ R_*L_{k+1} \xra{R_*d} R_*R_1L_k \xra{i_*} R_{*+1}L_k,$$
where $d: L_{k+1} \ra R_1L_k$ is nonzero.
\end{cor}

\subsection{Analyzing $s_{k*}$} \label{sk section}

Analyzing $s_{k*}$ is more delicate, as $s_k: QL_1(k) \ra QL_1(k+1)$ is not assumed to be an infinite loop map.

Recall that the adjoint to $s_k$ has components $s_k(r): \Sinfty D_r L_1(k) \ra \Sinfty L_1(k+1)$, and \lemref{decomp lemma} tells us that $s_k$ is the sum over $r$ of the maps $\Oinfty s_k(r)\circ j_r$.
The next lemma, an application of Weiss orthogonal calculus \cite{weiss}, shows that, in our situation, almost all of these are zero.

\begin{lem} \label{Weiss calc lemma} If $s_k$ extends to a natural transformation $s_k(V): QL_1(k,V) \ra QL_1(k+1,V)$, then $s_k(r)$ is null for $r> p$.
\end{lem}
\begin{proof}  Under the hypothesis of the lemma, each $s_k(r)$ extends to a natural transformation $s_k(r,V): \Sinfty D_rL_1(k,V) \ra \Sinfty L_1(k+1,V)$.  As functors of $V$ to spectra, $\Sinfty D_rL_1(k,V)$ is homogeneous of degree $rp^k$.  Now apply the standard fact: if $F$ is a homogeneous functor of degree $d$, there are no nontrivial natural transformations from $F$ to any polynomial functor of degree less that $d$.
\end{proof}

Recall that only the terms $\Oinfty s_k(p^m)_*\circ j_{p^m*}$ contribute to $s_{k_*}$ when restricted to primitives.  Thus the lemma implies that
$$ s_k* = \Oinfty s_k(1)_* + \Oinfty s_k(p)_*\circ j_{p*}.$$

The first term here is determined by $\tilde s_{k}(1)_*: L_{k} \ra R_*L_{k+1}$.  This is zero due to the next lemma, which is again a special case of \thmref{m not n thm}.

\begin{lem} \label{s lem 2} $\Hom_{\A}(L_{k},R_nL_{k+1}) = 0$ for all $n$.
\end{lem}

Thus $ s_k* = \Oinfty s_k(p)_*\circ j_{p*}$.  As made explicit in \corref{QZ to QY on primitives cor}, this is determined by
$ \tilde s_{k}(p)_*: R_1L_{k} \ra R_*L_{k+1}$,  and is thus constrained by the next two lemmas.

\begin{lem} \label{s lem 3} $\Hom_{\A}(R_1L_{k},R_nL_{k+1}) = 0$ for $n>0$.
\end{lem}

This is yet another special case of \thmref{m not n thm}.

\begin{lem} \label{s lem} $\Hom_{\A}(R_1L_k, L_{k+1}) \simeq \Z/p$.
\end{lem}

See \lemref{s lem 4} for a more precise statement.

\begin{cor}  \label{s cor} If $s_k$ satisfies condition $(2_k)$, then
$s_{k_*}$ is a composite of the form
$$ R_{*+1}L_k \xra{tr}  R_*R_1L_k \xra{R_*s} R_{*}L_{k+1},$$
where $s: R_1 L_k \ra L_{k+1}$ is nonzero.
\end{cor}

\begin{rem}  \label{Whitehead conjecture remark 2} With $s_{k*}$ defined as in \remref{Whitehead conjecture remark}, the conclusion of this corollary clearly holds.  This avoids the need to use \lemref{Weiss calc lemma}, and thus Weiss calculus, in a proof of the Whitehead conjecture.
\end{rem}

\subsection{What is next in the proof of \thmref{prim thm}?}

The analysis above focuses attention on the $\A$--module maps among the modules of the form $R_nL_k$.  In the next four sections we describe these in a manner that allow us to easily finish the proof of \thmref{prim thm}.

The key points are the following:
\begin{itemize}
\item $R_nL_k$ is a direct summand of $R_1^{n+k}L_0$.
\item $\Hom_{\A}(R_1^mL_0,R_1^nL_0) = 0$ for $m< n$.
\item $\End_{\A}(R_1^nL_0) = \HH_n$, where $\HH_n$ is the Hecke algebra of type $A_{n-1}$.
\item Thanks to the first three points, the maps occurring in \corref{d cor} and \corref{s cor} can be described in terms of elements in these Hecke algebras.
\end{itemize}

Some of this material is in \cite{kuhn4} or \cite{kuhnpriddy}. We also make use of special properties of $H^*(BE_n)$, viewed as an unstable $\A$--module.

\section{The Hecke algebra and Steinberg idempotent for $GL_n(\Z/p)$}

In this and subsequent sections, we often write $GL_n$ for $GL_n(\Z/p)$.

We describe some subgroups of $GL_n$, viewed as a Chevalley group of type $A_{n-1}$.

We let $B_n < GL_n$ be the standard Borel subgroup: the set of all upper triangular matrices.  This contains the subgroup $U_n$ consisting of all upper triangular matrices with 1's on the main diagonal.  This is a $p$--Sylow subgroup of $GL_n$.

For $i=1, \dots, n-1$, we let $P_i < GL_n$ be the $i$th minimal parabolic subgroup: the subgroup generated by $B_n$ and the transposition $w_i$ that interchanges coordinates $i$ and $i+1$ of $E_n$. These transpositions generate the group $W_n$ of permutation matrices.  As is standard, given $w \in W_n$, we let $l(w)$ be the minimal length of a word in the $w_i$ representing $w$.

\subsection{The Hecke algebra of type $A_{n-1}$.}

\begin{defn} For $n\geq 1$, the mod $p$ Hecke algebra of type $A_{n-1}$ is the endomorphism algebra $$\mathcal H_n = \End_{GL_n}(\Z/p[B_n\backslash GL_n]).$$

\end{defn}
It is convenient to also let $\mathcal H_0 = \Z/p$.

Recall that if $M$ is a left $\Z/p[G]$--module, and $H<G$, then
$$M_H = \Z/p[H \backslash G] \otimes_{\Z/p[G]} M$$ is the module of coinvariants.

Standard Yoneda lemma arguments show the following.

\begin{lem} $\mathcal H_n$ is the endomorphism ring of the functor $M \mapsto M_{B_n}$.
\end{lem}

This lemma allows us to define some elements in $\mathcal H_n$. (We do this in a way that makes it clear these elements really come from the integral Hecke algebra: see \cite[\S 4]{kuhn4} and \cite[\S 2]{kuhnpriddy} for more about this.)

To describe these, recall that, given $K<H<G$, and a $\Z/p[G]$--module $M$, there is an evident quotient map $M_K \ra M_H$. There is also a transfer (or norm) map $tr: M_H \ra M_K$ defined by $\displaystyle tr([x]) = \sum_{i}[g_ix]$ if $\displaystyle \coprod_i g_iK = H$.  The composite $M_H \xra{tr} M_K \ra M_H$ is multiplication by the index $[H:K]$.

\begin{defn}  For $i=1, \dots, n-1$, let $\hat e(i), e(i) \in \mathcal H_n$ be defined as follows.  We let $\hat e(i)$ be the natural transformation
$$ M_{B_n} \ra M_{P_i} \xra{tr} M_{B_n}.$$
As the index of $B_n$ in $P_i$ is $p+1$, $\hat e(i)^2 = (p+1)\hat e(i)$, and so  is idempotent mod $p$.  We then let $e(i) = p+1-\hat e(i) = 1-\hat e(i) \mod p$, the complementary idempotent.
\end{defn}

In terms of the standard algebra generators \cite{iwahori} $T_1, \dots ,T_{n-1}$ for $\mathcal H_n$, $\hat e(i) = 1+ T_i$, corresponding to the Bruhat decomposition $P_i = B_n \coprod B_nw_iB_n$.  Thus $e(i) = p-T_i = -T_i \mod p$.

The equation $\hat e(i)^2 = (p+1)\hat e(i)$ is then equivalent to the classic relation
$$T_i^2 = (p-1)T_i + p.$$
The classic presentation \cite{iwahori} for $\mathcal H_n$ in terms of the $T_i$ gives the following
presentation of $\mathcal H_n$ in terms of the $e(i)$.

\begin{prop} The $\Z/p$--algebra $\mathcal H_n$ is generated by $e(1), \dots, e(n-1)$ subject to the relations
\begin{itemize}
\item[(i)] $e(i)^2 = e(i)$ for all $i$.
\item[(ii)] $e(i)e(i+1)e(i)= e(i+1)e(i)e(i+1)$ for $1\leq i\leq n-2$.
\item[(iii)] $e(i)e(j)=e(j)e(i)$ if $j-i>1$.
\end{itemize}

Furthermore, the assignment $e(i) \mapsto \hat e(i)$ extends to an involution of algebras $$\hat{\text{}}: \mathcal H_n \ra \mathcal H_n.$$
\end{prop}

\begin{rem} In the integral Hecke algebra, relations (i) and (ii) lift to
\begin{itemize}
\item[(i)] $e(i)^2 = (p+1)e(i)$.
\item[(ii)]$e(i)e(i+1)e(i) + pe(i)= e(i+1)e(i)e(i+1) + pe(i+1)$.
\end{itemize}
\end{rem}

\begin{defn/prop} \cite[\S 4]{kuhn4} There is a unique nonzero element $e_n \in \mathcal H_n$ satisfying the two properties:
\begin{itemize}
\item[(a)] $e(i)e_n = e_n = e_ne(i)$ for all $i$.
\item[(b)] $e_n^2=e_n$.
\end{itemize}
\end{defn/prop}

Note that $e_1 = e_0 = 1$.  For larger $n$, there is an explicit formulae for $e_n$.

\begin{prop} For $n\geq 2$, $e_n$ is the `longest word' in the $e(i)$.
\end{prop}

\begin{ex}  In $\mathcal H_4$, $e_4 = e(1)e(2)e(3)e(1)e(2)e(1)$.
\end{ex}

\begin{cor} For all $\mathcal H_n$--modules $N$, $\displaystyle e_nN = \bigcap_{i=1}^{n-1} e(i)N$.
\end{cor}

Using the involution of $\mathcal H_n$, one sees that the element $\hat e_n \in \mathcal H_n$ satisfies analogous properties.  ($\hat e_1 = \hat e_0 = 1$.) The `hatted' version of the corollary for $\mathcal H_n$--modules of the form $M_{B_n}$ is as follows.

\begin{cor} For all $\Z/p[GL_n]$--modules $M$, multiplication by $\hat e_n$ is the composite $M_{B_n} \ra M_{GL_n} \xra{tr} M_{B_n}$.
\end{cor}
\begin{proof} Since $\displaystyle \im \{M_{GL_n} \xra{tr} M_{B_n}\} \subseteq \bigcap_{i=1}^{n-1} \im \{M_{P_i} \xra{tr} M_{B_n}\}$, the natural transformation $M_{B_n} \ra M_{GL_n} \xra{tr} M_{B_n}$ is an idempotent satisfying properties that characterize $\hat e_n$.
\end{proof}

\begin{notation} There is an algebra embedding $\mathcal H_k \times \mathcal H_n \hra \mathcal H_{k+n}$ corresponding to the block inclusion $GL_k \times GL_n \hra GL_{k+n}$. Explicitly this sends $e(i) \in \mathcal H_k$ to $e(i) \in \mathcal H_{k+n}$ and $e(i) \in \mathcal H_n$ to $e(k+i) \in \mathcal H_{k+n}$.  Slightly abusing notation, in $\mathcal H_{k+n}$, we let $e_k, \hat e_n \in \mathcal H_{k+n}$ be the images of $e_k \in \mathcal H_k$ and $\hat e_n \in \mathcal H_n$ under this embedding. Thus $e_k$ and $\hat e_n$ are commuting idempotents, $e_k$ is the longest word in the first $k-1$ of the $e(i)$,  $\hat e_n$ is the longest word in the last $n-1$ of the $\hat e(i)$, and neither involves $e(k)$.
\end{notation}

A useful consequence of the above propositions goes as follows.

\begin{cor} \cite{kuhnpriddy} \label{ek cor} In $\mathcal H_{k+n}$, $e_ke(k)e_{k}=e_{k+1}$ and $\hat e_n \hat e(k) \hat e_n = \hat e_{n+1}$.
\end{cor}

\subsection{The Steinberg idempotent}

\begin{defn}  In the integral group ring $\Z[GL_k(\Z/p)]$, let $\displaystyle b_k = \sum_{b \in B_k} [b]$, and let $\displaystyle \tilde w_k = \sum_{w \in W_k} (-1)^{l(w)}[w]$.  Then let $e^{St}_k \in \Z_{(p)}[GL_k(\Z/p)]$ be
$$ e^{St}_k = \frac{1}{[Gl_k:U_k]}\widetilde w_k b_k.$$
\end{defn}

R.~Steinberg \cite{steinberg} showed that $e^{St}_k$ is idempotent.

As $e^{St}_kb = e^{St}_k$ for all $b \in B_k$, for any $\Z_{(p)}[GL_k]$--module $M$, $e^{St}_k: M \ra M$ factors through the quotient $M \ra M_{B_k}$.

\begin{prop} \cite{kuhn4} For all $\Z/p[GL_k]$--modules $M$, the natural composition $M_{B_k} \xra{e^{St}_k} M \ra M_{B_k}$ is $e_k \in \mathcal H_k$.
\end{prop}

It follows that there are canonical isomorphisms $e^{St}_kM \simeq e_kM_{B_k}$.

\begin{prop} \label{ek st prop} \cite{kuhnpriddy} $e_k^{St}e_{k+1}^{St}=e_{k+1}^{St} = e_{k+1}^{St}e_k^{St} \in \Z_{(p)}[GL_{k+1}(\Z/p)]$.
\end{prop}

Here $\Z_{(p)}[GL_{k}(\Z/p)]$ is viewed as a subalgebra of $\Z_{(p)}[GL_{k+1}(\Z/p)]$ using the inclusion $GL_k \times GL_1 \subset GL_{k+1}$.

\subsection{Using $e^{St}_k$ to define the co-H-space $L_1(k)$.}

As $GL_k(\Z/p)$ acts on the suspension space $BE_k^{\rho_k}$, one can find a map $e^{St}_k: BE_k^{\rho_k} \ra BE_k^{\rho_k}$ which will induce multiplication by $e^{St}_k$ on homology.  The methods of \cite{adamskuhn} show that there is even a choice of such a map that is homotopy idempotent.

The space $L_1(k)$ is then defined to be the mapping telescope $Tel(e^{St}_k)$, and the canonical map $r: BE_k^{\rho_k} \ra L_1(k)$ fits into a homotopy commutative diagram
\begin{equation*}
\xymatrix{
L_1(k) \ar[dr]^i \ar@{=}[rr] && L_1(k) \ar[dr]^i & \\
& BE_k^{\rho_k} \ar[rr]^{e^{St}_k} \ar[ur]^r && BE_k^{\rho_k} }
\end{equation*}

\section{Some subgroups of $\Sigma_{p^n}$ and wreath product transfers}

Using iterated wreath products, the symmetric group $\Sigma_{p^n}$ contains subgroups analogous to the subgroups $B_n, P_i < GL_n$, as we now describe.

Our notational convention for wreath products is that if $H$ permutes a set $S$ and $G$ permutes a set $T$, then $H \wr G = G^S \rtimes H$ permutes the set $T \times S$, and contains $G \times H$ as a subgroup.

We now let $\widetilde B_n$ be the $n$--fold iterated wreath product $\Sigma_p \stackrel{n}{\wr} \Sigma_p$ and $\widetilde P_i$ be the $n-1$--fold wreath product $\widetilde B_{n-i-1} \wr \Sigma_{p^2} \wr \widetilde B_{i-1}$.  $\widetilde B_n$ has $\widetilde U_n = \Z/p \stackrel{n}{\wr} \Z/p$ as a Sylow subgroup.  Note that $E_n$ is canonically a subgroup of all of these.

In \secref{introduction} we mentioned that $N_{\Sigma_{p^n}}(E_n)/E_n = GL_n$. Similarly, one has
\begin{lem} \cite[Prop.3.7]{kuhn4}
$N_{\widetilde B_n}(E_n)/E_n = B_n$, $N_{\widetilde P_i}(E_n)/E_n = P_i$, and \\ $N_{\widetilde U_n}(E_n)/E_n = U_n$.
\end{lem}

It follows that the subgroup inclusions $E_n < \widetilde B_n$ and $E_n < \widetilde P_i$ induce maps in homology passing through the $B_n$ and $P_i$ coinvariants:
$$ H_*(BE_n^{\rho_n})_{B_n} \ra H_*(B\widetilde B_n^{\rho_n})$$
and
$$ H_*(BE_n^{\rho_n})_{P_i} \ra H_*(B\widetilde P_i^{\rho_n}).$$

The main theorem in \cite{kuhn4} shows that cohomology diagrams dual to the ones in following theorem commute.

\begin{thm} \label{transfer thm}  The diagrams
$
\xymatrix{
H_*(BE_n^{\rho_n})_{P_i} \ar[d] \ar[r]^-{tr} &  H_*(BE_n^{\rho_n})_{B_n}\ar[d]  \\
H_*(B\widetilde P_i^{\rho_n}) \ar[r]^-{tr} & H_*(B\widetilde B_n^{\rho_n}) }
$
commute.
\end{thm}

Now we note that the Thom spaces $B\widetilde B_n^{\rho_n}$ and $B\widetilde P_i^{\rho_n}$ identify as $D_p^nS^1$ and  $D_p^{i-1}D_{p^2}D_p^{n-i-1}S^1$.  Thus \lemref{Rn lem 2} tells us that
$$R_1^nL_0 = \im \{ \tilde H_*(BE_n^{\rho_n})_{B_n} \ra \tilde H_*(B\widetilde B_n^{\rho_n})\}$$
and
$$R_1^{n-i-1}R_2R_1^{i-1}L_0 = \im \{ \tilde H_*(BE_n^{\rho_n})_{P_i} \ra \tilde H_*(B\widetilde P_i^{\rho_n})\}.$$

\begin{cor} \label{transfer cor} The diagrams
$
\xymatrix{
\tilde H_*(BE_n^{\rho_n})_{P_i} \ar@{->>}[d] \ar[r]^-{tr} &  \tilde H_*(BE_n^{\rho_n})_{B_n}\ar@{->>}[d]  \\
R_1^{n-i-1}R_2R_1^{i-1}L_0 \ar[r]^-{tr} & R_1^nL_0 }
$
commute.
\end{cor}

Mitchell and Priddy, using a homology calculation similar to those in \thmref{transfer thm}, observe the following.

\begin{prop} \cite{MiP83}  There exists a stable map $f_k: \Sinfty B\widetilde B_k^{\rho_k} \ra \Sinfty BE_k^{\rho_k}$ such that the diagram
\begin{equation*}
\xymatrix{
& \Sinfty B\widetilde B_k^{\rho_k} \ar[dr]^{f_k}  &\\
\Sinfty BE_k^{\rho_k} \ar[rr]^{\Sinfty e^{St}_k} \ar[ur]&& \Sinfty BE_k^{\rho_k} }
\end{equation*}
homotopy commutes.  Thus $L_1(k)$ is a stable retract of $B\widetilde B_k^{\rho_k} = D_p^kS^1$, and the $\A$--module $L_k$ is a direct summand of $R_1^kL_0$.
\end{prop}

\begin{cor} \label{DpLk retract cor} $L_1(k+1)$ is a stable retract of $D_pL_1(k)$, and the $\A$--module $L_{k+1}$ is a direct summand of $R_1L_k$.
\end{cor}
\begin{proof}  This follows from the commutative diagram
\begin{equation*}
\xymatrix{
BE_{k+1}^{\rho_{k+1}} \ar[d]_{e^{St}_{k+1}} \ar[dr]^{e^{St}_{k+1}} &&  \\  BE_{k+1}^{\rho_{k+1}} \ar[d] \ar[r]^{e^{St}_{k}\times 1} &  BE_{k+1}^{\rho_{k+1}} \ar[d] \ar[r] &  B\widetilde B_{k+1}^{\rho_{k+1}} \ar@{=}[d]\\
 D_pBE_{k}^{\rho_{k}} \ar[r]^{D_p e^{St}_{k}}& D_pBE_{k}^{\rho_{k}} \ar[r] &  D_pB\widetilde B_k^{\rho_k}}
\end{equation*}
where the unlabeled maps are induced by subgroup inclusions.  The triangle commutes by \propref{ek st prop}.
\end{proof}

\section{$\A$--module maps among the modules $R_nL_k$}

Recall that we wish to understand $\A$--module maps among the $\A$--modules
$$R_nL_k \subseteq H_*(D_{p^n}L_1(k)).$$

The results in the last sections make it clear that very explicit geometric maps exhibit $R_nL_k$ as a direct summand of $R_1^{k+n}L_0$.  Thus it suffices to understand $\A$--module maps among the family $R_1^nL_0$.

Recall that $E_n < \widetilde B_n$ induces an epimorphism of $\A$--modules
$$ H_*(BE_n^{\rho_n})_{B_n} \twoheadrightarrow R_1^nL_0 \subseteq H_*(B\widetilde B_n^{\rho_n}).$$

\begin{rem} The epimorphism here is an isomorphism when $p=2$, and very close to an isomorphism when $p$ is odd: see \lemref{Nil iso lem}.
\end{rem}

\begin{thm} \label{m not n thm} If $m< n$,
$$\Hom_{\A}(H_*(BE_m^{\rho_m})_{B_m},H_*(BE_n^{\rho_n})_{B_n}) = \Hom_{\A}(R_1^mL_0, R_1^nL_0) =
0.$$
Thus $\Hom_{\A}(R_mL_j,R_nL_k)= 0$ if $m+j< n+k$.
\end{thm}

We will prove this in \secref{End RnL0 section}. Note that this includes Lemmas \ref{d lem 1}, \ref{s lem 2}, and \ref{s lem 3} as special cases.

More interesting is when $m=n$.

Let $\End_{\A}(H_*(BE_n^{\rho_n})_{B_n}, R_1^nL_0)$ be the set of pairs $(f,g)$ of $\A$--module maps making the diagram
\begin{equation*}
\xymatrix{
H_*(BE_n^{\rho_n})_{B_n} \ar@{->>}[d] \ar[r]^f & H_*(BE_n^{\rho_n})_{B_n} \ar@{->>}[d]  \\
R_1^nL_0 \ar[r]^g & R_1^nL_0 }
\end{equation*}
commute.  This set is a subalgebra of $\End_{\A}(H_*(BE_n^{\rho_n})_{B_n})$ in the evident way, and there is an algebra homomorphism to $\End_{\A}(R_1^nL_0)$ by sending $(f,g)$ to $g$.

Note that there is a homomorphism of algebras
$$ \mathcal H_n \ra \End_{\A}(M_{B_n})$$
for all $\A$--modules $M$ equipped with an action of $GL_n$.

\begin{thm} \label{End thm}  The algebra homomorphisms
$$ \mathcal H_n \ra \End_{\A}(H_*(BE_n^{\rho_n})_{B_n}) \supseteq \End_{\A}(H_*(BE_n^{\rho_n})_{B_n}, R_1^nL_0) \ra End_{\A}(R_1^nL_0)$$
are all isomorphisms.
\end{thm}

This fundamental result will also be proved in \secref{End RnL0 section}.

It is worth pondering what key elements in $\mathcal H_n$ correspond to in $End_{\A}(R_1^nL_0)$.

Firstly $R_1^nL_0 \subset H_*(D_p^nS^1) = H_*(B\widetilde B_n^{\rho_n})$ can be viewed as the span of sequences of Dyer Lashof operations of length $n$ acting on the fundamental class of $S^1$, {\em before} Adem relations have been applied.

\corref{transfer cor} tells us that $\hat e(i) \in \mathcal H_n$ corresponds to the endomorphism
$$ R_1^nL_0 \xra{i_*} R_1^{n-i-1}R_2R_1^{i-1}L_0 \xra{tr} R_1^nL_0$$
where $i_*$ and $tr$ are induced by the wreath product subgroup inclusion $\widetilde B_n < \widetilde P_i$.  This map can be regarded as `rewriting' the $i$th and $(i+1)$st operations (counting from the right) in some sort of `admissible' form.

The complementary idempotent $e(i)$ rewrites these operations in `completely inadmissible' form.

We see that $\hat e_nR_1^nL_0 = R_nL_0$ and $e_kR_1^kL_0 = L_k$, and finally that
$$ \hat e_ne_kR_1^{n+k}L_0 = R_nL_k, $$
which can be interpreted as the span of admissible sequences of length $n$ applied to the module of completely inadmissible sequences of length $k$ (acting on a 1--dimensional class).

\begin{cor} \label{End cor} If $m+j=n+k$,
$ \Hom_{\A}(R_mL_j, R_nL_k) = \hat e_ne_k\mathcal H_{k+n} \hat e_m e_j$.
\end{cor}

\subsection{More calculations in the Hecke algebra}

The next lemma and corollary should be compared with \lemref{d lem} and \corref{d cor}.

\begin{lem} \label{d lem 2}
$ \Hom_{\A}(L_k, R_1L_{k-1}) = \hat e_1e_{k-1}\mathcal H_{k}  e_k = \langle e_k \rangle$.
\end{lem}

\begin{cor} \label{d cor 2} If $d: L_k \ra R_1L_{k-1}$ corresponds to $\lambda e_k$ for some $\lambda \in \Z/p$, then the composite
$$ R_nL_k \xra{R_nd} R_nR_1L_{k-1} \xra{i_*} R_{n+1}L_{k-1}$$
corresponds to $\lambda \hat e_{n+1}e_k \in \hat e_{n+1}e_{k-1}\mathcal H_{k+n} \hat e_n e_k$.
\end{cor}

Similarly, the next lemma and corollary should be compared with \lemref{s lem} and \corref{s cor}.

\begin{lem} \label{s lem 4}
$ \Hom_{\A}(R_1L_{k-1}, L_k) = e_k\mathcal H_{k}  \hat e_1e_{k-1} = \langle e_k \rangle$.
\end{lem}

\begin{cor} \label{s cor 2} If $s: R_1L_{k-1}\ra L_k$ corresponds to $\lambda e_k$ for some $\lambda \in \Z/p$, then the composite
$$ R_{n+1}L_{k-1} \xra{tr} R_nR_1L_{k-1} \xra{R_ns} R_nL_k$$
corresponds to $\lambda e_k \hat e_{n+1} \in \hat e_n e_k\mathcal H_{k+n} \hat e_{n+1}e_{k-1}$.
\end{cor}

The identity in the next proposition was identified by us in \cite{kuhnpriddy}.

\begin{prop} In \ $\End_{\A}(R_nL_k) = \hat e_ne_k\mathcal H_{k+n} \hat e_{n} e_{k}$,
$$ \hat e_{n} e_{k+1}\hat e_n +e_k \hat e_{n+1} e_{k} = \hat e_n e_k.$$
The elements $\hat e_{n} e_{k+1}\hat e_n$ and $e_k \hat e_{n+1} e_{k}$ are orthogonal, and thus idempotent.
\end{prop}

\begin{proof}  For the first statement, we compute:
\begin{equation*}
\begin{split}
\hat e_{n} e_{k+1}\hat e_n +e_k \hat e_{n+1} e_{k} &
= \hat e_{n} e_ke(k)e_k\hat e_n +e_k \hat e_n \hat e(k) \hat e_n e_{k} \\
  & = \hat e_{n} e_k(e(k)+ \hat e(k)) \hat e_n e_{k}\\
  & = (\hat e_{n} e_k)^2 = \hat e_n e_{k}\\
\end{split}
\end{equation*}
Similarly, for the second statement, we have:
\begin{equation*}
\begin{split}
(\hat e_{n} e_{k+1}\hat e_n)(e_k \hat e_{n+1} e_{k}) &
= \hat e_{n} e_{k+1} e_k \hat e_n \hat e_{n+1} e_{k} \\
  & = \hat e_{n} e_{k+1} \hat e_{n+1} e_{k}\\
  & = \hat e_{n} e_{k+1} e(k) \hat e(k) \hat e_{n+1} e_{k}\\
  & = \hat e_{n} e_{k+1} 0 \hat e_{n+1} e_{k} = 0.
\end{split}
\end{equation*}
It is now formal that each of these elements are idempotent: if $x,y$ are elements in an algebra with unit $e$ such that $x+y=e$ and $xy=0$, then $x$ and $y$ are idempotent.
\end{proof}

\begin{cor} \label{iso cor} If $\lambda, \mu \in \Z/p$ are nonzero, then
$$\lambda \hat e_{n} e_{k+1}\hat e_n + \mu e_k \hat e_{n+1} e_{k}: R_nL_k \ra R_nL_k$$ will be an isomorphism.
\end{cor}

\section{Proofs of \thmref{m not n thm} and \thmref{End thm}} \label{End RnL0 section}

To prove \thmref{m not n thm} and \thmref{End thm}, it is convenient to work with cohomology, rather than homology.  The inclusion $E_n < \widetilde B_n$ induces
$$ \widetilde H^*(B\widetilde B_n^{\rho_n}) \twoheadrightarrow (R_1^nL_0)^* \hookrightarrow \widetilde H^*(BE_n^{\rho_n})^{B_n}.$$

It is also convenient to desuspend once, using that the representation $\rho_n$ has a trivial 1--dimensional summand.  Let $N(n) = \Sigma^{-1} (R_1^nL_0)^*$. If one writes $\rho_n = 1 \oplus \tilde \rho_n$, then $B\widetilde B_n^{\rho_n} = \Sigma B\widetilde B_n^{\tilde \rho_n}$, $BE_n^{\rho_n} = \Sigma BE_n^{\tilde \rho_n}$, and the inclusion $E_n < \widetilde B_n$ induces
$$ \widetilde H^*(B\widetilde B_n^{\tilde \rho_n}) \twoheadrightarrow N(n) \hookrightarrow \widetilde H^*(BE_n^{\tilde \rho_n})^{B_n}.$$

We let $c_n \in H^{p^n-1}(BE_n)$ denote the Euler class of the bundle associated to $\tilde \rho_n$.  When $p=2$, this is the top Dickson invariant; when $p$ is odd, it is the `square root' of the top Dickson invariant of the `even part' of $H^*(BE_n)$.  In either case, $\widetilde H^*(BE_n^{\tilde \rho_n}) = c_nH^*(BE_n)$, and so $N(n) \subseteq c_nH^*(BE_n)^{B_n}$.  (By $c_nH^*(BE_n)^{B_n}$, we mean $(c_nH^*(BE_n))^{B_n}$: at odd primes, $c_n$ is a twisted invariant of $B_n$.)

We will use the next property of $c_n$.

\begin{lem} \label{cn lem} If $E<E_n$ is a proper subgroup, then $c_n$ restricts to 0 in $H^*(BE)$.
\end{lem}

To compare $N(n)$ and $c_nH^*(BE_n)^{B_n}$, we will use the following lemma, whose proof will be given at the end of the section.

\begin{lem} \label{Nil iso lem} $N(n) \subseteq c_nH^*(BE_n)^{B_n}$ is an isomorphism if $p=2$, and has nilpotent cokernel if $p$ is odd.
\end{lem}

Here `nilpotent' is in the sense of unstable $\A$--module theory \cite{schwartz}.  Let $\U$ be the category of unstable modules, and $\K$ be the category of unstable algebras. In our case, our modules are in $\K$, and the lemma equivalently says that a sufficiently high $p$th power of any element in $c_nH^*(BE_n)^{B_n}$ will be in the submodule $N(n)$.

Since $\Hom_{\U}(N(n),N(m)) \subseteq \Hom_{\U}(N(n),c_mH^*(BE_m)^{B_m})$,
to prove \thmref{m not n thm} it suffices to prove the next proposition.

\begin{prop} \label{m not n prop} If $m< n$, then
$$ \Hom_{\U}(N(n),c_mH^*(BE_m)) = \Hom_{\U}(c_nH^*(BE_n),c_mH^*(BE_m)) = 0.$$
\end{prop}

Similarly, to prove \thmref{End thm}, it suffices to prove the following proposition.

\begin{prop} \label{End prop} The homomorphisms
$$ \mathcal H_n \ra \End_{\U}(c_nH^*(BE_n)^{B_n}) \ra \Hom_{\U}(N(n),c_nH^*(BE_n)^{B_n}) \supseteq \End_{\U}(N(n))$$
are all isomorphism.
\end{prop}

We begin the proofs.

Recall that, by either \cite{agm} or \cite{lz}, the natural map
$$\Z/p[\Hom(E_m,E_n)] \ra \Hom_{\U}(H^*(BE_n),H^*(BE_m))$$ is an isomorphism.  Our first lemma is a variant of this.

\begin{lem}  The natural map
$$\Z/p[\Epi(E_m,E_n)] \ra \Hom_{\U}(c_nH^*(BE_n),H^*(BE_m))$$ is an isomorphism.
\end{lem}

\begin{proof} Consider the diagram
\begin{equation*}
\xymatrix{
\Z/p[\Hom(E_m,E_n)] \ar@{->>}[r] \ar[d]^-{\wr} & \Z/p[\Epi(E_m,E_n)]  \ar@{^{(}->}[d]\\
\Z/p[\Hom_{\K}(H^*(BE_n),H^*(BE_m))] \ar[r]  \ar[d]^-{\wr}& \Z/p[\Hom_{\K}(c_nH^*(BE_n),H^*(BE_m))] \ar[d]^-{\wr}\\
\Hom_{\U}(H^*(BE_n),H^*(BE_m)) \ar@{->>}[r] & \Hom_{\U}(c_nH^*(BE_n),H^*(BE_m)).}
\end{equation*}

The top horizontal map sends any homomorphism $E_m \ra E_n$ that is not epic to 0, and the top square commutes because such a homomorphism has image in a proper subgroup of $E_n$, so that \lemref{cn lem} applies.  The bottom two vertical maps are isomorphisms by Lannes' Linearization Theorem \cite[Appendix]{lz}.  The bottom horizontal map is an epimorphism because $H^*(BE_m)$ is a $\U$--injective \cite{schwartz}.

Finally, the top right vertical map is visibly monic: seen, e.g., by considering the action of $\Epi(E_m,E_n)$ on $c_mH^1(BE_m)$.  From the rest of the diagram, one concludes that this same map is an epimorphism, and the lemma follows.
\end{proof}

\begin{cor} $\Z/p[\Epi(E_m,E_n)/B_n] \xra{\sim} \Hom_{\U}(c_nH^*(BE_n)^{B_n},H^*(BE_m)) \xra{\sim} \Hom_{\U}(N(n),H^*(BE_m))$.
\end{cor}
\begin{proof} The first isomorphism follows from the lemma, again using that $H^*(BE_n)$ is a $\U$--injective.  The second follows from \lemref{Nil iso lem}, using, for odd primes, that $H^*(BE_n)$ is nil--closed \cite{schwartz}.
\end{proof}

As there are no epimorphisms $E_m \ra E_n$ if $m<n$, this corollary proves \propref{m not n prop} and thus \thmref{m not n thm}.

We turn to the case when $m=n$.  The next lemma and corollary imply \propref{End prop}.
\begin{lem} With $M$ equal to $N(n)$ or $c_nH^*(BE_n)^{B_n}$,
$$ \mathcal H_n = \Hom_{\U}(M,c_nH^*(BE_n)^{B_n}).$$
\end{lem}
\begin{proof} Note that $\mathcal H_n = \Z/p[GL_n/B_n]^{B_n}$.
We have already shown that the natural map
$$ \Z/p[GL_n/B_n] \ra \Hom_{\U}(M,H^*(BE_n))$$
is a bijection. By construction, the image of this map lands in the subspace $\Hom_{\U}(M,c_nH^*(BE_n))$.  Now take $B_n$ invariants.
\end{proof}

\begin{cor} $\End_{\U}(N(n)) \subseteq \Hom_{\U}(N(n),c_nH^*(BE_n)^{B_n})$ is an isomorphism.
\end{cor}
\begin{proof}  As discussed earlier in dual form, \thmref{transfer thm} tells us that there are maps $e(i): N(n) \ra N(n)$ making the diagram
\begin{equation*}
\xymatrix{
N(n) \ar@{^{(}->}[d] \ar@{-->}[r]^{e(i)} & N(n) \ar@{^{(}->}[d]  \\
c_nH^*(BE_n)^{B_n} \ar[r]^{e(i)} & c_nH^*(BE_n)^{B_n} }
\end{equation*}
commute.  Since the $e(i)$ generate the algebra $\End_{\U}(c_nH^*(BE_n))^{B_n})$, the corollary follows.
\end{proof}

It remains to prove \lemref{Nil iso lem}.

\begin{proof}[Sketch proof of \lemref{Nil iso lem}]

The lemma says that the restriction
$$ H^*(B\widetilde B_n^{\tilde \rho_n}) \ra H^*(BE_n^{\tilde \rho_n})^{B_n}.$$
is an epimorphism when $p=2$ and a nil--epimorphism in all cases.

This can be proved in various ways. For starters, using a transfer argument, one can replace $\widetilde B_n$ and $B_n$ with their $p$--Sylow subgroups $\widetilde U_n$ and $U_n$.  This makes orientability issues disappear: the Euler class $c_n$ extends to $\widetilde U_n$ and is $U_n$--invariant.

One approach now goes as follows. The map
$$ c_nH^*(B\widetilde U_n) \ra c_nH^*(BE_n^{\tilde \rho_n})^{U_n}.$$
is a nil--epimorphism if and only if, for all elementary abelian $p$--groups $E$,
$$ \Hom_{\U}(c_nH^*(BE_n^{\tilde \rho_n})^{U_n}, H^*(BE)) \ra \Hom_{\U}(c_nH^*(B\widetilde U_n), H^*(BE))$$
is a monomorphism.

One then checks that this map identifies with
$$ \Z/p[\Epi(E,E_n)/U_n] \ra \Z/p[\TransRep(E,\widetilde U_n)],$$
where $\TransRep(E,\widetilde U_n)$ is set of representations of $E$ in $\widetilde U_n$ that are transitive, when viewed as an $E$--set with $p^n$ elements.  This is a monomorphism.
\end{proof}

\section{Proof of the main results} \label{final proofs section}

We complete the proof of \thmref{prim thm} and thus also \thmref{big thm}.

\begin{proof}[Proof of \thmref{prim thm}: the last step]

We need to understand the diagram
$$ R_*L_{k+1} \begin{array}{c} d_{k*} \\[-.08in] \longrightarrow \\[-.1in] \longleftarrow \\[-.1in] s_{k*}
\end{array} R_*L_k \begin{array}{c} d_{{k-1}*} \\[-.08in] \longrightarrow \\[-.1in] \longleftarrow \\[-.1in] s_{{k-1}*}
\end{array} R_*L_{k-1}$$
if conditions $(1_j)$ and $(2_j)$ hold for $j=k-1,k$.

By \corref{d cor} and \corref{s cor}, this diagram is the direct sum over $n$ of diagrams of the form
$$ R_{n-1}L_{k+1} \begin{array}{c} d_{k*} \\[-.08in] \longrightarrow \\[-.1in] \longleftarrow \\[-.1in] s_{k*}
\end{array} R_nL_k \begin{array}{c} d_{{k-1}*} \\[-.08in] \longrightarrow \\[-.1in] \longleftarrow \\[-.1in] s_{{k-1}*}
\end{array} R_{n+1}L_{k-1}.$$

\corref{d cor} and \corref{d cor 2} combine to say that $d_{k*}$ and $d_{{k-1}*}$ are nonzero multiples of $\hat e_{n}e_{k+1}$ and $\hat e_{n+1}e_{k}$, while \corref{s cor} and \corref{s cor 2} combine to say that $s_{k*}$ and $s_{{k-1}*}$ are nonzero multiples of $e_{k+1}\hat e_{n}$ and $e_k\hat e_{n+1}$.

Thus there are nonzero elements $\lambda, \mu \in \Z/p$ such that
$$ {d_k}_*  {s_k}_* + {s_{k-1}}_* {d_{k-1}}_*: R_nL_k \ra R_nL_k$$
equals
$$\lambda (\hat e_{n}e_{k+1})(e_{k+1}\hat e_{n}) + \mu (e_k\hat e_{n+1})(\hat e_{n+1}e_{k}) = \lambda \hat e_{n}e_{k+1}\hat e_{n} + \mu e_k\hat e_{n+1}e_{k}:R_nL_k \ra R_nL_k,$$
which is an isomorphism, by \corref{iso cor}.

\end{proof}

We now prove \lemref{d_k lemma}, showing that the maps $d_k$ in the Whitehead conjecture satisfy the conditions $(1_k)$.

\begin{proof}[Proof of \lemref{d_k lemma}] Recall that one has the identification $$\Sinfty L_1(k) = \Sigma^{1-k} SP^{p^k}(S)/SP^{p^{k-1}}(S).$$  We need to be prove that
$ d_k:QL_1(k+1) \ra QL_1(k)$
is nonzero on $H_{c(k+1)}$, where $d_k$ is the composition $\Oinfty \delta_k \circ \Oinfty i_{k+1}$    in the diagram
\begin{equation*}
\xymatrix{
& \Oinfty \Sigma^{-k} H\Z/SP^{p^k}(S) \ar[dr]^{\Oinfty \delta_k} && \Oinfty \Sigma^{1-k} H\Z/SP^{p^{k-1}}(S) \\
QL_1(k+1) \ar[ur]^{\Oinfty i_{k+1}} & & QL_1(k) \ar[ur]^{\Oinfty i_{k}} &}
\end{equation*}

We prove this by induction on $k$, beginning with the trivial and slightly degenerate case $d_{-1}=\Oinfty i_0: QS^1 \ra S^1$, which is nonzero on $H_1$.

For the inductive step, when the conditions of \thmref{big thm} hold for $j<k$, \thmref{big thm} implies a truncated version of  \corref{cor 1}, and we learn that the fibration sequence
$$ \Oinfty \Sigma^{-k} H\Z/SP^{p^k}(S) \xra{\Oinfty \delta_k} QL_1(k) \xra{\Oinfty i_{k}} \Oinfty \Sigma^{1-k} H\Z/SP^{p^{k-1}}(S)$$
is a split fibration of infinite loop spaces.  Thus $\Oinfty \delta_k$ is certainly monic in homology.

Meanwhile, $\Oinfty i_{k+1}$ is an isomorphism on bottom homology $H_{c(k+1)}$, as this is the case for $i_{k+1}$.
\end{proof}

We prove \lemref{s_k lemma}, showing that the maps $s_k$ arising from the Goodwillie tower of $S^1$ satisfy the conditions $(2_k)$.

\begin{proof}[Proof of \lemref{s_k lemma}] We need to prove that
$ s_k:QL_1(k) \ra QL_1(k+1)$ is nonzero on $PH_{c(k+1)}$.  The right vertical map is an isomorphism in the diagram
\begin{equation*}
\xymatrix{
\pi_{c(k+1)}(QL_1(k)) \ar[d] \ar[r]^-{s_{k*}} & \pi_{c(k+1)}(QL_1(k+1)) \ar[d]  \\
PH_{c(k+1)}(QL_1(k)) \ar[r]^-{s_{k*}} & PH_{c(k+1)}(QL_1(k+1)).}
\end{equation*}
Thus we need to prove that $ s_k:QL_1(k) \ra QL_1(k+1)$ is nonzero on $\pi_{c(k+1)}$.

We prove this by induction on $k$, beginning with the trivial and slightly degenerate case $s_{-1}:S^1 \ra QS^1$, which is nonzero on $\pi_1$.

We now make use of the fact that $s_{k*}$ appears as a differential on the $E^1$ page of the homotopy spectral sequence which we know, thanks to \cite{aronekan}, strongly converges to $\pi_*(S^1)$.  This second quadrant spectral sequence has $E^1_{-s,t} = \pi_{t}^S(L_1(s))$, and converges to $\pi_{t-s}(S^1)$, so that $E^{\infty}_{-s,t} =0$ unless $(s,t)=(0,1)$.

For our inductive step, when the conditions of \thmref{big thm} hold for $j<k$, \thmref{big thm} implies a truncated version of \corref{cor 2}, and we learn that
$$ 0 \ra \pi_*(S^1) \ra \pi_*(QL_1(0)) \ra \cdots \ra \pi_*(QL_1(k-1)) \ra \pi_*(QL_1(k))$$
is exact. Interpreted as a statement about the spectral sequence, we learn that
$E^2_{-s,t} = 0$ for $s=0, \dots, k-1$ unless $(s,t)=(0,1)$.

By connectivity reasons, if $s >k+1$, $E^1_{-s,t}=0$ unless $t-s$ is much bigger than $c(k+1)-k-1$.  Thus, since $E^{\infty}_{-(k+1),c(k+1)}=0$, the $E^1$--differential
$$ E^1_{-k,c(k+1)} \ra E^1_{-(k+1),c(k+1)}$$
must be onto.  But this just means that
$$s_{k*}:\pi_{c(k+1)}(QL_1(k)) \ra \pi_{c(k+1)}(QL_1(k+1))$$
is onto, as needed.
\end{proof}

\appendix
\section{Proof of \corref{cor 1}} \label{appendix 1}

\corref{cor 1} follows from the next proposition, specialized to the situation of \thmref{big thm}.

\begin{prop}  Suppose one has a diagram of spaces of the following form:
\begin{equation} \label{diagram for prop}
\xymatrix{
\dots \ar@<.5ex>[r]^{d_2} & X_2 \ar@<.5ex>[r]^{d_1} \ar@<.5ex>[l]^{s_2} & X_1 \ar@<.5ex>[r]^{d_0} \ar@<.5ex>[l]^{s_1} & X_0 \ar@<.5ex>[l]^{s_0} \ar@<.5ex>[r]^{d_{-1}} & Y_{-1}. \ar@<.5ex>[l]^{s_{-1}} \\}
\end{equation}
Suppose the following properties hold: \\

\noindent{\bf (i)} $d_{k-1}d_{k}$ is null for all $k$. \\

\noindent{\bf (ii)} The $d_k$ are maps of H--spaces. (Infinite loop, in our situation.) \\

\noindent{\bf (iii)} The maps $d_{-1}s_{-1}: Y_{-1} \ra Y_{-1}$ and
$ d_ks_k + s_{k-1}d_{k-1}: X_k \ra X_k$,
for $k \geq 0$, are all homotopy equivalences.\\

Then (\ref{diagram for prop}) refines to a diagram
\begin{equation*}
\xymatrix{
\dots \ar@<.5ex>[rr]^{d_2} \ar[dr]^{p_2} && X_2 \ar@<.5ex>[rr]^{d_1} \ar@<.5ex>[ll]^{s_2} \ar[dr]^{p_1}&& X_1 \ar@<.5ex>[rr]^{d_0} \ar@<.5ex>[ll]^{s_1} \ar[dr]^{p_0}&& X_0 \ar@<.5ex>[ll]^{s_0} \ar@<.5ex>[rr]^{\epsilon} && Y_{-1}. \ar@<.5ex>[ll]^{\eta} \\
& Y_2 \ar[ur]^{i_2}&& Y_1 \ar[ur]^{i_1} && Y_0 \ar[ur]^{i_0} &&&}
\end{equation*}
with the following properties. \\

\noindent{\bf (a)} \  $Y_{k} \xra{i_{k}} X_k \xra{p_{k-1}} Y_{k-1}$ is a split fibration sequence of $H$--spaces, and thus define a product decomposition $X_k \simeq Y_k \times Y_{k-1}$. \\

\noindent{\bf (b)} \ The triangles
$ \xymatrix{
X_{k+1} \ar[dr]^{p_k} \ar[rr]^{d_k} && X_{k}  \\
& Y_k \ar[ur]^{i_k}& }
$
commute.

\noindent{\bf (c)} \ For all $k\geq 0$, $p_ks_k i_k: Y_k \ra Y_k$ is a homotopy equivalence.
\end{prop}
\begin{proof}  One constructs the refinement by induction on $k$, so assume the properties for $k-1$.  One then defines $i_k: Y_k \ra X_k$ to be the fiber of $p_{k-1}$. Property (c) tells us that $i_{k-1}$ is a monomorphism in the sense of category theory.  To define $p_{k}$, one knows that $i_{k-1}p_{k-1}d_{k} = d_{k-1} d_{k}$ is null, by assumption (i).  Thus so is $p_{k-1}d_{k}$, and so the lift $p_{k}$ exists as in property (b).  Property (c) then holds for the new triangle by assumption (iii), showing that $p_k$ is split.  Assumption (ii) tells us that the split fibrations $$Y_{k} \xra{i_{k}} X_k \xra{p_{k-1}} Y_{k-1}$$ are product decompositions.
\end{proof}

\section{Discussion of \conjref{loop prop}} \label{loop prop appendix}

With our notation, Theorem 6.3 of \cite{aronedwyerlesh} says the following.

\begin{thm} \label{adl thm}  For any $d>0$, the $d$--fold looping map from
$$ \hoNat[ QBE_k^{(\R \oplus V) \otimes_{\R} \rho_k}, QBE_{k+1}^{(\R \oplus V) \otimes_{\R} \rho_{k+1}}]$$
to
$$ \hoNat[ \Omega^dQBE_k^{(\R \oplus V) \otimes_{\R} \rho_k}, \Omega^dQBE_{k+1}^{(\R \oplus V) \otimes_{\R} \rho_{k+1}}]$$
is onto on $\pi_0$.
\end{thm}

Here $\hoNat$ denotes the (homotopically meaningful) space of natural transformations.

This theorem has the formal consequence that there exist natural $k$--fold deloopings $s_k(V): QL_1(k,V) \ra QL_1(k+1,V)$ of the $k$th connecting map in the Weiss tower.

We observe that their argument allows one to show some uniqueness.

\begin{thm} \label{loop thm}  For any $d>1$, the $(d-1)$--fold looping map from
$$ \hoNat[ \Omega QBE_k^{(\R \oplus V) \otimes_{\R} \rho_k}, \Omega QBE_{k+1}^{(\R \oplus V) \otimes_{\R} \rho_{k+1}}]$$
to
$$ \hoNat[ \Omega^dQBE_k^{(\R \oplus V) \otimes_{\R} \rho_k}, \Omega^dQBE_{k+1}^{(\R \oplus V) \otimes_{\R} \rho_{k+1}}]$$
is bijective (and, in particular, one-to-one) on $\pi_0$.
\end{thm}

\begin{cor}  There is a unique natural $(k-1)$--fold delooping
$$\Omega s_k(V): \Omega QL_1(k,V) \ra \Omega QL_1(k+1,V)$$ of the $k$th connecting map of the Weiss tower.
\end{cor}

One might conjecture the following generalizations of these theorems.

\begin{conj} \label{adl conj}  For all $r\geq 1$, and for any $d>0$, the $d$--fold looping map from
$$ \hoNat[ QBE_k^{(\R \oplus V) \otimes_{\R} \rho_k}, QBE_{k+r}^{(\R \oplus V) \otimes_{\R} \rho_{k+r}}]$$
to
$$ \hoNat[ \Omega^dQBE_k^{(\R \oplus V) \otimes_{\R} \rho_k}, \Omega^dQBE_{k+r}^{(\R \oplus V) \otimes_{\R} \rho_{k+r}}]$$
is onto on $\pi_0$.
\end{conj}

\begin{conj} \label{loop conj}  For all $r\geq 1$, for any $d>1$, the $(d-1)$--fold looping map from
$$ \hoNat[ \Omega QBE_k^{(\R \oplus V) \otimes_{\R} \rho_k}, \Omega QBE_{k+r}^{(\R \oplus V) \otimes_{\R} \rho_{k+r}}]$$
to
$$ \hoNat[ \Omega^dQBE_k^{(\R \oplus V) \otimes_{\R} \rho_k}, \Omega^dQBE_{k+r}^{(\R \oplus V) \otimes_{\R} \rho_{k+r}}]$$
is bijective (and, in particular, one-to-one) on $\pi_0$.
\end{conj}

This second conjecture, in the case when $r=2$, implies \conjref{loop prop}: the conjecture that $\Omega s_k(V)\Omega s_{k-1}(V)$ would be null.

The analysis in \cite{aronedwyerlesh} proving \thmref{adl thm} also proves \thmref{loop thm} if one makes one change:
\begin{itemize}
\item Near the bottom of page 202 of \cite{aronedwyerlesh}, the (0--connected) cofiber of the map $S^0 \ra S^{(i-1)d}$ should be replaced by the (at least 1--connected) cofiber of  $S^{i-1} \ra S^{(i-1)d}$.
\end{itemize}

To prove the conjectures, one is led to the attempt to do the following.

\begin{itemize}
\item $k+1$ should be replaced by $k+r$ in the discussion beginning on page 203 of \cite{aronedwyerlesh}, as well as in supporting arguments in \S 5.
\end{itemize}

All of this appears to be straightforward, except for generalizing the `representation theory' described in Lemma 5.1 and Proposition 5.2 of \cite{aronedwyerlesh}.

As a first observation, one can weaken Lemma 5.1 to just what is needed in \cite[\S 6]{aronedwyerlesh}.  We recall the context.

Let $1 \leq i \leq p^r$.

There is a subgroup inclusion $U(ip^k) \times U(p^{k+r} - ip^k) \subset U(p^{k+r})$ using block matrices.  In the usual way, we consider $\Sigma_i \wr E_k$ to be a subgroup of $U(ip^k)$, and $E_{k+r}$ to be a subgroup of $U(p^{k+r})$.  ($E_k$ here is written as $\Delta_k$ in \cite{aronedwyerlesh}.)

Now fix $g \in U(p^{k+r})$, and let
$$I_g = \{(\phi,\gamma) \in \Sigma_i \wr E_k \times E_{k+r} \ | \ \exists h \in U(p^{k+r} - ip^k) \text{ such that }g\gamma g^{-1}=\phi h\}.$$

When $r=1$, the next lemma is a version of Lemma 5.1 (5-1) of  \cite{aronedwyerlesh} sufficient for the needs of the proof of both theorems above.  It is proved using the argument in \cite{aronedwyerlesh}.

\begin{lem} \label{appendix lem 1} The composite $I_g \subset \Sigma_i \wr E_k \times E_{k+r} \ra E_{k+r}$ is monic.  In particular, $I_g$ will be an elementary abelian $p$--group.
\end{lem}

The evident generalization of Lemma 5.1 (5-2) of  \cite{aronedwyerlesh} would be the following: the composite $I_g \subset \Sigma_i \wr E_k \times E_{k+r} \ra \Sigma_i \wr E_k$ is monic.

Much to our disappointment, we found this can be false as soon as $r=2$.

\begin{ex}  Let $p=r=i=2$ and $k=1$.  Then, for $g \in U(8)$, $I_g$ is a subgroup of $\Sigma_2\wr E_1 \times E_3$.  If $g$ is chosen to conjugate the permutation matrix $\gamma$ consisting of four $2\times 2$ blocks $\begin{bmatrix} 0&1 \\ 1&0  \end{bmatrix}$ to the matrix with $4\times 4$ blocks $\begin{bmatrix} I_4&\mathbf 0 \\ \mathbf 0&-I_4  \end{bmatrix}$, then $(e,\gamma) \in I_g$ (with $h=-I_4$).

Thus the finiteness statement, analogous to that on page 204 of \cite{aronedwyerlesh}, that one would like to make about the spectrum $S^{-4}\sm S^8$ as an $I_g$--spectrum appears to be not true.
\end{ex}

\end{document}